\documentclass[reqno,10pt]{amsart}
\usepackage{amsmath, amsthm}
\usepackage{pdfsync}
\usepackage{parskip}
\usepackage{hyperref}
\usepackage{mathtools}
\usepackage{mdframed}
\usepackage{mathrsfs}
\usepackage{tikz}
\numberwithin{equation}{section}
\usepackage[margin=1.2in]{geometry}

\newtheorem{theorem}{Theorem}[section]

\newtheorem{lem}[theorem]{Lemma}
\newtheorem{cor}[theorem]{Corollary}
\newtheorem{example}[theorem]{Example}
\theoremstyle{remark}
\newtheorem{remark}[theorem]{Remark}
\newtheorem{rmk}[theorem]{Remarks}

\def\M{\mathsf{M}}
\def\p{\mathsf{p}}
\def\x{x_{\mathsf{p}}}
\def\Xp{X_{\mathsf{p}}}
\def\Vp{V_{\mathsf{p}}}

\def\R{\mathbb{R}}
\def\C{\mathbb{C}}

\def\H{\mathbb{H}^{m}}

\def\N{\mathbb{N}}
\def\Nz{\mathbb{N}_0}
\def\A{\mathfrak{A}}
\def\K{\frak{K}}
\def\J{\mathbb{J}}

\def\L{\mathcal{L}}

\def\bK{\mathbb{K}}

\def\B{\mathbb{B}}

\def\Ok{\mathsf{O}_{\kappa}}
\def\Otk{\widetilde{\mathsf{O}}_{\kappa}}
\def\Ohk{W_{\kappa}}
\def\Ohkc{W_{\kappa,c}}
\def\vpk{\varphi_{\kappa}}
\def\psk{\psi_{\kappa}}

\def\Bm{\mathbb{B}^{m}}

\def\kb{\varphi^{\ast}_{\kappa}}
\def\kf{\psi^{\ast}_{\kappa}}

\def\cA{\mathcal{A}}
\def\bfa{\boldsymbol{\mathfrak{a}}}
\def\Ce{\mathcal{C}_{\frak{e}}}

\def\CMB{{manifolds with bounded geometry}}
\def\URM{{\em uniformly regular Riemannian manifold}}
\def\URMs{{\em uniformly regular Riemannian manifolds}}

\begin{document}


\pagespan{1}{}

\title[Some Remarks on Uniformly Regular Riemannian Manifolds]{Some Remarks on Uniformly Regular Riemannian Manifolds}

\author[M. Disconzi]{Marcelo Disconzi}
\address{Department of Mathematics\\
         Vanderbilt University \\
         Nashville, TN~37240, USA}
\email{marcelo.disconzi@vanderbilt.edu}
\author[Y. Shao]{Yuanzhen Shao}
\address{Department of Mathematics,
         Purdue University, 
         150 N. University Street, 
         West Lfayette, IN 47906, USA}
\email{shao92@purdue.edu}
\author[G. Simonett]{Gieri Simonett}
\address{Department of Mathematics\\
         Vanderbilt University \\
         Nashville, TN~37240, USA}
\email{gieri.simonett@vanderbilt.edu}

\thanks{The research of the first author is supported by NSF grant PHY-1305705, and the research of the third author is supported by NSF DMS-1265579.}


\begin{abstract}
We establish the equivalence between the family of  uniformly regular Riemannian manifolds without boundary and the class of manifolds with bounded geometry. 
\end{abstract}
\maketitle

\section{\bf Introduction}

In 2012, H.~Amann introduced a class of (possibly noncompact) manifolds, called {\URMs}. Roughly speaking, an $m$-dimensional Riemannian manifold $(\M,g)$ is {\em uniformly regular} if its differentiable structure is induced by an atlas such that all its local patches are of approximately the same size, all derivatives of the transition maps are bounded, and the pull-back metric of $g$ in every local coordinate is comparable to the Euclidean metric $g_m$. The precise definition of {\URMs} will be presented in Section~2 below.

In the sequel, we understand all our manifolds to be smooth and without boundary, unless  stated otherwise.
The main objective of this short note is to prove that the family of {\URMs} coincides with the class of {\CMB}. A  manifold is said to be of bounded geometry if it has positive injectivity radius, and all covariant derivatives of the curvature tensor are bounded. 
The precise definition of bounded geometry will be given later in this introductory section. 

Nowadays, there is rising interest in studying differential equations on non-compact manifolds, see \cite{HiePru97, MaAn99, MazNis06, Shi89, Shi90, Zhang97, Zhang00}, for instance. It is a well-known fact that many well established analytic tools in Euclidean space fail, in general, on Riemannian manifolds. For instance, for $u \in C^2(\R^m)$ with $u^*:=\sup u<\infty$, 
we can always find a sequence $(x_k)_{k\in\N} \subset\R^m$ such that
$$u(x_k)>u^*-1/k,\quad |\nabla u (x_k)|<1/k,\quad  \Delta u (x_k) <1/k.$$
Nevertheless, it is known that this maximum principle does not always hold true on non-compact Riemannian manifolds. Indeed, there are counterexamples provided by H.~Omori \cite{Omo67}.  

A cornerstone in the study of differential equations is the theory of function spaces. 
In order to study this theory on Riemannian manifolds, it is natural to impose extra geometric conditions, most likely certain restrictions on the curvatures. 
Among all efforts made to  find proper assumptions, one extensively studied category is the class of {\CMB}.  
A manifold $(\M,g)$ is said to have a positive injectivity radius if there exists a positive number $\iota(\M)$ such that the exponential map at $\p\in\M$, $exp_\p$, is a diffeomorphism from 
$$\Bm(0,\iota(\M)):=\{\xi\in \R^m: |\xi|_{g_m} <\iota(\M) \}\quad \text{onto} \quad O_\p(\iota(\M)):=exp_\p(\Bm(0,\iota(\M))) $$ 
for all $\p\in\M$. 
A manifold $(\M,g)$ is of bounded geometry if it has a positive injectivity radius and all covariant derivatives of the curvature tensor are bounded, i.e.,
\begin{align}
\label{S1: def bdd geo}
\| |\nabla_g^k R|_g \|_\infty \leq C(k), \quad k\in\Nz,
\end{align}
where $R$ is the Riemannian curvature tensor, and $\nabla_g$ denotes the extension of the Levi-Civita connection over $\mathcal{T}^\sigma_\tau\M$ with $\sigma,\tau\in\Nz$. 
Let $T{\M}$ and $T^{\ast}{\M}$ be the tangent and the cotangent bundle, respectively.
Then $\mathcal{T}^\sigma_\tau\M$ is the $C^\infty(\M)$-module of all smooth sections of $T^\sigma_\tau\M:= T\M^{\otimes \sigma}\otimes T^*\M^{\otimes \tau}$, the $(\sigma,\tau)$-tensor bundles. $|\cdot|_g$ is the (vector bundle) norm induced by the extension of the Riemannian metric $g$ from the tangent bundle to $T^{\sigma}_{\tau}{\M}$. Condition~\eqref{S1: def bdd geo} can be replaced by the boundedness of all covariant derivatives of the sectional curvature. See \cite[Section~5.14]{BishGold80} for a justification. This condition can also be formulated equivalently by asking that in all normal geodesic coordinates of every local chart $(O_\p(\iota(\M)), exp_\p^{-1})$, we have
$$\det[(g_{ij})_{ij}]\geq c,\quad \max_{|\alpha|\leq k}\|\partial^\alpha g_{ij}\|_\infty\leq C(k), \quad k\in\Nz, $$
where $(g_{ij})_{ij}$ is the local matrix expression of the metric tensor $g$, for some constants $C(k),c>0$.  See \cite[Section~7.2.1]{Tri92}.
The amount of literature on differential equations on {\CMB} is vast. Most of the work concerns heat kernel estimates and spectral theory. See, for example, \cite{Dav89, GriSal09} and the references therein. With additional restrictions like nonnegative Ricci curvature, $L_q$-$L_p$ maximal regularity theory is established for second order elliptic operators. See \cite{HiePru97, MazNis06}.

To illustrate some of our recent results for differential equations on {\URMs}, we look at linear differential operators $\cA:\mathcal{T}^\sigma_\tau \M \rightarrow \Gamma(\M, T^\sigma_\tau \M)$ of order $l$ acting on $(\sigma,\tau)$-tensor fields, defined by
\begin{equation*}
\cA=\cA(\bfa):=\sum\limits_{r=0}^l a_r \bullet (\nabla_g^r \cdot),
\end{equation*}
where $a_r\in \Gamma(\M, T^{\sigma+\tau+r}_{\tau+\sigma} \M)$, the set of all sections of $T^{\sigma+\tau+r}_{\tau+\sigma} \M$, and $\bullet$ denotes complete contraction. See \cite[Section~2]{ShaoSim13} for a detailed discussion. 
We consider the following initial value problem on a  {\URM} $(\M,g)$ in  H\"older spaces:
\begin{equation}
\label{S1: IBVP}
\left\{\begin{aligned}
\partial_t u +\cA u &=f  &&\text{on}&&\M_T; \\
u(0)&=u_0  &&\text{on }&&\M .&&
\end{aligned}\right.
\end{equation}
Here $\M_T:= \M\times (0,T)$ for $T\in (0,\infty]$.

For $k\in{\N}_{0}$ and $\sigma,\tau\in\Nz$, we define $\|u\|_{k,\infty}:={\max}_{0\leq{i}\leq{k}}\||\nabla_g^{i}u|_{g}\|_{\infty}$ and
$$BC^{k}(\M, T^\sigma_\tau \M):=(\{u\in{C^k(\M, T^\sigma_\tau \M)}:\|u\|_{k,\infty}<\infty\},\|\cdot\|_{k,\infty}).$$
We also set
$
BC^\infty(\M, T^\sigma_\tau \M):=\bigcap_{k}BC^k(\M, T^\sigma_\tau \M),
$
endowed with the conventional projective topology. Then
\begin{center}
$bc^k(\M, T^\sigma_\tau \M):=$ the closure of $BC^\infty(\M, T^\sigma_\tau \M)$ in $BC^k(\M, T^\sigma_\tau \M)$.
\end{center}
Let $k<s<k+1$. Now the {\em little H\"older} space $bc^s(\M, T^\sigma_\tau \M)$ is defined by
\begin{align*}
bc^s(\M, T^\sigma_\tau \M):=(bc^k(\M, T^\sigma_\tau \M),bc^{k+1}(\M, T^\sigma_\tau \M))^0_{s-k,\infty}.
\end{align*}
Here $(\cdot,\cdot)^0_{\theta,\infty}$ is the continuous interpolation method, see \cite[Example~I.2.4.4]{Ama95}. 
In \cite{Ama13, AmaAr},  the theory of function spaces, including the {\em little H\"older} space, is studied.

A linear differential operator $\cA:=\cA(\bfa)$ is said to be {\em normally elliptic} if there exists some constant $\Ce>0$ such that for every pair $(\p,\xi)\in \M\times\Gamma(\M, T^\ast M)$ with $|\xi(\p)|_{g^*(\p)}= 1$ for all $\p\in\M$, the principal symbol of $\cA$ defined by
$$\hat{\sigma}\cA^\pi(\p,\xi(\p)):=(a_l\bullet(-i\xi)^{\otimes l})(\p)\in \L(T_\p\M^{\otimes\sigma}\otimes T_\p^*\M^{\otimes\tau})$$ 
satisfies 
$
S:=\Sigma_{\pi/2}:=\{z\in\C: |{\rm arg}z|\leq \pi/2 \}\cup\{0\} \subset \rho(-\hat{\sigma}\cA^\pi(\p,\xi(\p))),
$
and
\begin{equation*}
(1 +|\mu|) \|(\mu + \hat{\sigma}\cA^\pi(\p,\xi(\p)))^{-1}\|_{\L(T_p\M^{\otimes\sigma}\otimes T_p^*\M^{\otimes\tau})} \leq \Ce ,\quad \mu\in S. 
\end{equation*}
In the above, $g^*$ is the contravariant metric induced by $g$.
We readily check that a normally elliptic operator must be of even order. $\cA$ is called $s$-{\em regular} if 
\begin{equation*}
a_r\in bc^s(\M, T^{\sigma+\tau+r}_{\tau+\sigma} \M), \quad r=0,1,\cdots,l.
\end{equation*}
The following continuous maximal regularity theorem has been established by two of the authors. 
\begin{theorem}
[Y. Shao, G. Simonett \cite{ShaoSim13}]
\label{S1: cont-MR}
Let $(\M,g)$ be a  uniformly regular Riemannian manifold, $\gamma\in (0,1]$, and $s\notin\Nz$.
Suppose that $\cA$ is a $2l$-th order  normally elliptic  and  $s$-regular  differential operator  acting on $(\sigma,\tau)$-tensor fields. Then for any 
$$(f,u_0) \in C_{1-\gamma}([0,\infty); bc^s(\M, T^\sigma_\tau\M)) \times bc^{s+2l\gamma}(\M, T^\sigma_\tau\M), $$
equation~\eqref{S1: IBVP} has a unique solution
$$u\in C_{1-\gamma}([0,\infty); bc^{s+2l}(\M, T^\sigma_\tau\M))\cap C^1_{1-\gamma}([0,\infty); bc^s(\M, T^\sigma_\tau\M)). $$
Equivalently, $\cA$ generates an analytic semigroup on $bc^s(\M, T^\sigma_\tau\M)$ and has the property of continuous maximal regularity. 
\end{theorem}
Here $C^k_{1-\gamma}([0,\infty);X)$ is some weighted $C^k((0,\infty);X)$-space for a given Banach space $X$, see \cite[Section~3]{ShaoSim13} for a precise definition.

One may observe from the statement of Theorem~\ref{S1: cont-MR} that no additional geometric assumption is needed. So it generalizes the existing results on {\CMB}, see \cite{Dav89, GriSal09, Tri92}.
By means of results of G.~Da~Prato, P.~Grisvard \cite{DaPra79}, S.~Angenent \cite{Ange90} and P. Cl\'ement, G. Simonett \cite{CleSim01}, Theorem~\ref{S1: cont-MR} gives rise to existence and uniqueness of solutions to many quasilinear or even fully nonlinear differential equations on {\URMs}. See \cite{Shao13, ShaoSim13} for example.

A similar result to Theorem~\ref{S1: cont-MR} in an $L_p$-framework can be found in \cite{Ama13b} for second order initial boundary value problems.

In \cite[Theorem~4.1]{Ama14}, it is shown that a manifold with bounded geometry is {\em uniformly regular}. We aim at establishing the other implication, i.e., a {\URM} is of bounded geometry. In Section~3, we prove the following theorem.
\begin{theorem}
\label{S1:Main Thm}
Suppose that $(\M,g)$ is a uniformly regular Riemannian manifold. Then 
\begin{itemize}
\item[(a)] all covariant derivatives of the curvature tensor are bounded, i.e.,
$$ \||\nabla_g^k R|_g \|_\infty \leq C(k), \quad k\in\Nz, $$
where $R$ is the Riemannian curvature tensor;
\item[(b)] $(\M,g)$ has positive injectivity radius;
\item[(c)] $(\M,g)$ is complete.
\end{itemize}
\end{theorem}
\begin{remark}
\label{S1: RMK}
It is known that part (b) of Theorem~\ref{S1:Main Thm} actually implies (c). See \cite[Proposition~1.2]{Eich07}.
We remind that a  manifold $(\M,g)$ is (geodesically) complete if all geodesics are infinitely extendible with respect to arc length. The Hopf-Rinow theorem states that this is equivalent to asserting that $\M$ is complete as a metric space with respect to the intrinsic metric induced by $g$.
We find it illustrative to present two direct proofs of completeness of {\URMs} that do not use positive injectivity radius a priori.
\end{remark}

\begin{cor}
A Riemannian manifold is  uniformly regular, if and only if it  has a bounded geometry.
\end{cor}

This article is organized as follows. In Section~2, we give the precise definition of {\URMs} and present several examples of this class. This concept can be extended to manifolds with boundary, see \cite{Ama13, AmaAr}. In Section~3, we provide a proof for Theorem~\ref{S1:Main Thm}.

\section{Uniformly regular Riemannian manifolds}

Let $(\M,g)$ be a  $C^{\infty}$-Riemannian manifold of dimension $m$ endowed with $g$ as its Riemannian metric such that its underlying topological space is separable. An atlas $\A:=(\Ok,\vpk)_{\kappa\in \K}$ for $\M$ is said to be normalized if $\vpk(\Ok)=\Bm$.
Here $\Bm$ is the unit ball centered at the origin in $\R^m$. We put $\psk:=\vpk^{-1}$. 

The atlas $\A$ is said to have \emph{finite multiplicity} if there exists $N\in \N $ such that any intersection of more than $N$ coordinate patches is empty. 
Put
\begin{equation*}
\mathfrak{N}(\kappa):=\{\tilde{\kappa}\in\mathfrak{K}:\mathsf{O}_{\tilde{\kappa}}\cap\Ok\neq\emptyset \}.
\end{equation*}
The {\em finite multiplicity} of $\A$ and the separability of $\M$ imply that $\A$ is countable.

An atlas $\A$ is said to fulfil the \emph{uniformly shrinkable} condition, if it is normalized and there exists $r\in (0,1)$ such that $\{\psk(r{\Bm}):\kappa\in\K\}$ is a cover for ${\M}$. 

Following H.~Amann \cite{Ama13, AmaAr}, we say that $(\M,g)$ is a {\URM} if it admits an atlas $\A$ such that
\begin{itemize}
\item[(R1)] $\A$ is uniformly shrinkable and has finite multiplicity. If $\M$ is oriented, then $\A$ is orientation preserving.
\item[(R2)] $\|\varphi_{\eta}\circ\psk \|_{k,\infty}\leq c(k) $, $\kappa\in\K$, $\eta\in\mathfrak{N}(\kappa)$, and $k\in{\N}_0$.
\item[(R3)] $\kf g\sim g_m $, $\kappa\in\K$. Here $g_m$ denotes the Euclidean metric on ${\R}^m$ and $\kf g$ denotes the pull-back metric of $g$ by $\psk$.
\item[(R4)] $\|\kf g\|_{k,\infty}\leq c(k)$, $\kappa\in\K$ and $k\in\Nz$.
\end{itemize}
Here $\|u\|_{k,\infty}:=\max_{|\alpha|\leq k}\|\partial^{\alpha}u\|_{\infty}$, and it is understood that a constant $c(k)$, like in (R2), depends only on $k$. An atlas $\A$ satisfying (R1) and (R2) is called a {\em uniformly regular atlas}. (R3) reads as
\begin{align}
\label{S2: R3}
|\xi|^2/c\leq \kf g(x)(\xi,\xi) \leq{c|\xi|^2}, \quad \text{for any }x\in \Bm,\, \xi\in \R^m, \,  \kappa\in\K \text{ and some }c\geq 1.
\end{align}

In the following, we will present several examples of {\URMs}.
\begin{example}
\label{S2: exp of URM}
\begin{itemize}
\item[]{\phantom{ some  text to complete some  } }
\item[(a)] Suppose that $(\M,g)$ and $(\tilde{\M},\tilde{g})$ are both uniformly regular Riemannian manifolds. Then so is $(\M\times\tilde{\M}, g+\tilde{g})$.
\item[(b)] Let $f: \tilde{\M}\rightarrow\M$ be a diffeomorphism of manifolds. If $(\M,g)$ is a uniformly regular Riemannian manifold, then so is
$(\tilde{\M},f^*g)$.
\item[(c)] $(\R^m,g_m)$  is uniformly regular.
\item[(d)] Every compact  manifold is  uniformly regular.
\item[(e)] Every manifold with bounded geometry is uniformly regular.
\item[(f)] Let $J:=(1,\infty)$, and $R_\alpha(t):J\rightarrow J$: $t\mapsto t^\alpha$ for $\alpha\in [0,1]$. Assume that $(B,g_B)$ is a $d$-dimensional compact submanifold of $\R^n$. The (model) $(R_\alpha,B)$-funnel $F(R_\alpha,B)$ on $J$ is defined by 
\begin{equation*}
F(R_\alpha,B)=F(R_\alpha,B;J):=\{(t,R_\alpha(t)y):\, t\in J, \, y\in B\}\subset \R^{1+n} .
\end{equation*}
It is a $(1+d)$-dimensional submanifold of $\R^{1+n}$. The map
\begin{equation*}
\phi: F \rightarrow J\times B: \quad (t,R_\alpha(t)y)\rightarrow (t,y) 
\end{equation*}
is a diffeomorphism. 
Suppose that $\{V_0,V_1\}$  is an open covering of $(\M,g)$ such that $(V_1,g)$ is isometric to some $(R_\alpha,B)$-funnel, $(F(R_\alpha,B),\phi^*(dt^2 + g_B))$, and $V_0$, $V_0\cap V_1$ are relatively compact in $\M$. 
Then $(\M,g)$ is uniformly regular.
\item[(g)] $(\M,g)$ is a stretched corner manifold, that is, $\{V_0,V_1\}$  is an open covering of $(\M,g)$ with $V_0$, $V_0\cap V_1$ relatively compact in $\M$, and $(V_1,g)$ is isometric to $( C(B),  (dt/t)^2 +g_B + (ds/(ts)^2)$. 
Here $(B,g_B)$ is an $d$-dimensional compact submanifold embedded  in the unit sphere of $\R^{d+1}$. 
Both $dt^2$  and $ds^2$ denote the standard metric of $(0,1]$.  
The  stretched (model) corner end $C(B)$ is defined by
\begin{equation*}
C(B):=\{ (ts, tsy ,s): (t,y,s)\in(0,1]\times B \times (0,1] \}\subset \R^{d+2}.
\end{equation*}
Then $(\M,g)$ is uniformly regular.
\item[(h)] The open unit ball $\B^m$ in the Euclidean space $\R^m$ equipped with the Poincar\'e disk metric, that is,
\begin{equation*}
(\B^m, 4dx^2/(1-|x|^2)^2),
\end{equation*}
is uniformly regular. 
\item[(i)] $(\M,g)$ is a ``$b$-Riemannian manifold". To be more precise, let $B$ be a compact manifold with boundary $\partial B$. We extend $\partial B$ by a stretched conic end, that is, $\M:=B\cup X$, where $X$ is diffeomorphic to $(0,1]\times \partial B$ and equipped with the metric $g=(dt/t)^2 +g_{\partial B}$. This metric $g$ is called an exact $b$-metric. Then $(\M,g)$ is uniformly regular.
\end{itemize}
\end{example}
With the help of the first two examples, we can construct more complicated {\URMs} based on Example~\ref{S2: exp of URM}(c)-(i).

In order to prove some of the statements in Examples~\ref{S2: exp of URM}, it is convenient to introduce the concept of {\em singular manifolds}. Roughly speaking, a manifold $(M, \hat{g})$ is called a {\em singular manifold} in the sense defined in \cite[Section~2]{Ama13}, if there exists some $\rho\in C^\infty(\M,(0,\infty))$ such that $(\M,\hat{g}/\rho^2)$ is {\em uniformly regular}. The function $\rho$ is called a {\em singularity datum} of $(\M,\hat{g})$. A {\em singular manifold} is {\em uniformly regular} if $\rho\sim {\bf 1}_\M$. If two real-valued functions $f$ and $g$ are equivalent in the sense that $f/c\leq g\leq cf$ for some $c\geq 1$, then we write $f\sim g$.
\begin{rmk}
\begin{itemize}
\item[]{\phantom{ some  text to complete some  } }
\item[(i)]  The concept of stretched corner manifolds is used in \cite{ChenLiuWei14}. Note that the conventional corner manifolds, see \cite{Sch94}, are not {\em uniformly regular}, but indeed are {\em singular manifolds}.
\item[(ii)] The concept of ``$b$-Riemannian manifolds" was introduced by R.B. Melrose, see \cite[Chapter~2]{Mel93} for a detailed discussion. Example~\ref{S2: exp of URM}(f) implies that Theorem~\ref{S1: cont-MR} on {\URMs} can be considered as some extension of the theory of $b$-calculus. 
\end{itemize}
\end{rmk}

\begin{proof}
[{\bf Proof  of Example~\ref{S2: exp of URM}}]
\begin{itemize}
\item[]{\phantom{ some  text to complete some  } }
\end{itemize}
(a) Example~\ref{S2: exp of URM}(a) is a special case of \cite[Theorem~3.1]{Ama14}.

(b) \cite[Lemma~3.4]{Ama14} implies Example~\ref{S2: exp of URM}(b).

(c) Example~\ref{S2: exp of URM}(c) follows from \cite[formula~(3.3)]{Ama14}.

(d) See \cite[Corollary~4.3]{Ama14}.

(e) See \cite[Theorem~4.1]{Ama14}, and  also \cite[Lemma 2.2.6]{Aub82} and \cite{Eich91}. 

(f) See \cite[Theorem~1.2]{Ama14}. 

(g) \cite[Example~5.1, Lemmas~5.2 and 6.1]{Ama14} imply that 
\begin{equation*}
( C(B),  s^2 dt^2 +(ts)^2g_B + ds^2)
\end{equation*}
is a {\em singular manifold} with {\em singularity datum} $R(ts,tsy,s)=ts$. Then the assertion follows from  \cite[Lemma~3.3]{Ama14}.

(h) We put $J_1:=[0,1)$. By \cite[Lemma~5.2]{Ama14}, we have $((0,1]; dt^2)$ is a {\em singular manifold} with singular datum $R_2(t):=t^2$. \cite[Lemma~3.4]{Ama14} implies that $(J_1; dt^2)$  is a {\em singular manifold} with singular datum $1-R_2(t)$. The unit sphere $\mathbb{S}^{m-1}$ is {\em uniformly regular} by Example~\ref{S2: exp of URM}(d). Then by \cite[Theorems~3.1, 8.1 and Lemma~3.4]{Ama14}, 
\begin{equation*}
(\B^m, dt^2 + (1-R_2(t))^2 g_S)
\end{equation*} 
is a {\em singular manifold} with {\em singularity datum} $1-R_2(t)$, where $g_S$ is the metric on $\mathbb{S}^{m-1}$ induced by $g_m$. It is easy to see that $(\B^m , dt^2 + (1-R_2(t))^2 g_S)$ is diffeomorphic to $(\B^m, g_m)$. Now the statment of Example~\ref{S2: exp of URM}(h) follows from the definition of {\em singular manifolds} and \cite[Lemma~3.4]{Ama14}.

(i) To see that Example~\ref{S2: exp of URM}(i) holds true, by \cite[Lemma~3.3]{Ama14} we only need to show that $(X, (dt/t)^2+g_{\partial B})$ is {\em uniformly regular}. \cite[Example~5.1, Lemmas~5.2 and 6.1]{Ama14} imply that $(X,dt^2+t^2 g_{\partial B})$ is a {\em singular manifold} with {\em singularity datum} $R(t)=t$. By definition, $(X, (dt/t)^2+g_{\partial B})$ is {\em uniformly regular}.

\end{proof}



\section{Proof of the main theorem}

\begin{proof}
[{\bf Proof of Theorem~\ref{S1:Main Thm}(a)-(b)}]
\begin{itemize}
\item[]{\phantom{ some  text to complete some  } }
\end{itemize}

(a)  To prove part (a), we need to first introduce some notations.
Let $\sigma,\tau\in\Nz$. For abbreviation, we set $\J^{\sigma}:=\{1,2,\ldots,m\}^{\sigma}$, and $\J^{\tau}$ is defined alike. Given local coordinates $\varphi=\{x^1,\ldots,x^m\}$, $(i):=(i_1,\ldots,i_{\sigma})\in\J^{\sigma}$ and $(j):=(j_1,\ldots,j_{\tau})\in\J^{\tau}$, we set
\begin{align*}
\frac{\partial}{\partial{x}^{(i)}}:=\frac{\partial}{\partial{x^{i_1}}}\otimes\cdots\otimes\frac{\partial}{\partial{x^{i_{\sigma}}}}, \quad \partial^{(i)}:=\partial^{i_{1}}\circ\cdots\circ\partial^{i_{\sigma}},\quad  dx^{(j)}:=dx^{j_1}\otimes{\cdots}\otimes{dx}^{j_{\tau}}
\end{align*}
with $\partial^{i}=\frac{\partial}{\partial{x^i}}$. The local representation of 
$a\in \Gamma(\M,T^{\sigma}_{\tau}{\M})$ with respect to these coordinates is given by 
\begin{align*}
a=a^{(i)}_{(j)} \frac{\partial}{\partial{x}^{(i)}} \otimes dx^{(j)} 
\end{align*}
with coefficients $a^{(i)}_{(j)}$ defined on $\Ok$.
We can also extend the Riemannian metric $(\cdot|\cdot)_g$ from the tangent bundle to any $(\sigma,\tau)$-tensor bundle $T^{\sigma}_{\tau}{\M}$ such that $(\cdot|\cdot)_g:=(\cdot|\cdot)_{g^\tau_\sigma}:T^{\sigma}_{\tau}{\M}\times{T^{\sigma}_{\tau}{\M}}\rightarrow \bK $ by 
\begin{equation*}
(a|b)_g =g_{(i)(\tilde{i})} g^{(j)(\tilde{j})}	a^{(i)}_{(j)}  \bar{b}^{(\tilde{i})}_{(\tilde{j})}
\end{equation*}
in every coordinate with $(i),(\tilde{i})\in \J^\sigma$, $(j),(\tilde{j})\in \J^\tau$ and 
\begin{equation*} 
g_{(i)(\tilde{i})}:= g_{i_1}g_{\tilde{i}_1}\cdots g_{i_\sigma}g_{\tilde{i}_\sigma},\quad g^{(j)(\tilde{j})}:= g^{j_1} g^{\tilde{j}_1}\cdots g^{j_\tau} g^{\tilde{j}_\tau}.
\end{equation*} 
In addition,
\begin{center}
$|\cdot|_g:=|\cdot|_{g^\tau_\sigma}:\mathcal{T}^{\sigma}_{\tau}{\M}\rightarrow{C^{\infty}}({\M})$, $a\mapsto\sqrt{(a|a)_g}$
\end{center}
is called the (vector bundle) norm induced by $g$.

The Riemannian curvature tensor, in every coordinate, reads as
\begin{equation*}
R_{ijk}^l= \partial_i \Gamma_{jk}^l -\partial_j \Gamma_{ik}^l +\sum\limits_r (\Gamma_{jk}^r\Gamma_{ir}^l -\Gamma_{ik}^r \Gamma_{jr}^l),
\end{equation*}
where $R=R_{ijk}^l \frac{\partial}{\partial x^l}\otimes dx^i \otimes dx^j\otimes dx^k$, and $\Gamma_{ij}^k$ are the Christoffel symbols with respect to the Levi-Civita connection $\nabla_g$, 
see \cite[p. 134]{GalHilLaf04}. For $l\in \J^1$, $(i)\in\J^3$ and $(r)\in \J^k$,
\begin{align}
\label{S3: der of R}
\nabla_g^k R = (\nabla_{(r)} R^l_{(i)})\frac{\partial}{\partial x^l} \otimes dx^{(i)} \otimes dx^{(r)},
\end{align} 
where $\nabla_{(r)} R^l_{(i)}= \partial^{(r)} R^l_{(i)} + b^{l}_{(i;r)}$ with $b^{l}_{(i;r)}$ a linear combination of the elements of 
\begin{equation*}
\{\partial^\alpha R^{\tilde{l}}_{(\tilde{i})} : |\alpha|<k, \, \tilde{l}\in \J^1, \, (\tilde{i})\in\J^3\},
\end{equation*} 
the coefficients being polynomials in the derivatives of the Christoffel symbols of order at most $k-1-|\alpha|$. See \cite[formula~(3.18)]{Ama13}. It follows from \cite[formula~(3.19)]{Ama13} that
\begin{align}
\label{S3: Christoffel sym}
\| \kf \Gamma^l_{ij} \|_{k,\infty} \leq C(k), \quad k\in\Nz, \quad i,j,l\in \{1,2,\cdots,m\}.
\end{align} 
One can easily infer from (R4) that
\begin{align}
\label{S3: g^*}
\| \kf g^*\|_{k,\infty} \leq C(k),\quad \kappa\in\K, \quad k\in\Nz, 
\end{align}
where $g^*$ is the induced contravariant metric. Now combing with (R4), \eqref{S3: der of R}-\eqref{S3: g^*} imply the asserted statement.
\medskip

(b)
It is instructive to comment on the intuition behind our proof 
of positive injective radius, as well
as some of its geometric underpinnings. The breakdown of injectivity of the exponential map
at a point $\p$ will happen, as it is known, at the first time when any two geodesics
leaving $\p$ intersect each other. Whether or not this happens, and when it happens,
is dictated by the curvature of the manifold, which governs the focusing of such
geodesics, and with large curvature values contributing for potentially smaller injectivity radius. In our case, the uniform bounds on derivatives of the metric
assure uniform bounds on the curvature. Thus, intuitively, it is not unexpected that
uniformly regular manifolds posses a lower bound on the injectivity radius.
We can turn this idea into a more geometric appealing proof, which we now state.

The \emph{uniformly shrinkable} condition implies that there exists a number $r\in (0,1)$ such that $(\Otk)_\kappa:=(\psk(r\Bm))_\kappa$ still forms an open cover of $\M$. 

We claim that there exists some constant $c\geq 1$ such that 
\begin{align}
\label{S3: equivalent of topo}
\B_g(\psk(x_\kappa),\delta/c)\subset \psk(\Bm(x_\kappa,\delta))\subset \B_g(\psk(x_\kappa),c\delta)
\end{align}
for any $\kappa\in\K$, $x_\kappa\in r\Bm$, and $\delta<1-r$.
Here $\B_g(\p,{\rm R})$ denotes all the points on $\M$ with distance less than ${\rm R}$ to $\p$ with respect to the intrinsic metric. 
Indeed, put $\p_\kappa:=\psk(x_\kappa)$, and let 
$d_g(\cdot,\cdot)$ denote the distance function, that is,
\begin{equation*}
d_g(\p, {\sf q}):= \inf\{ L(\gamma): \gamma: [0,1]\to \M \text{ piecewise $C^\infty$-curve with } \gamma(0)=\p, \, \gamma(1)={\sf q} \}, 
\end{equation*}
where $L(\gamma)$ is the length of $\gamma$ with respect to the intrinsic metric induced by $g$,
see \cite[p. 15]{Jost11}. For any $\p\notin \psk(\Bm(x_\kappa,\delta))$ and on any piece-wise smooth curve $\gamma:[0,1]\rightarrow \M$ connecting $\p_\kappa$ and $\p$, i.e., $\gamma(0)=\p_\kappa$ and $\gamma(1)=\p$, we take $t_\gamma$ to be the first escape time of $\gamma$ out of $\psk(\Bm(x_\kappa,\delta))$. 
Then letting $\p_\gamma:=\gamma(t_\gamma)$, one can compute by means of (R3)
\begin{align*}
d_g(\p,\p_\kappa)&\geq \inf\limits_\gamma \int\limits_0^{t_\gamma} \sqrt{ g(\gamma^\prime(t),\gamma^\prime(t))}\, dt\\
&=\inf\limits_\gamma \int\limits_0^{t_\gamma} \sqrt{\kf g(\kf\gamma^\prime(t),\kf\gamma^\prime(t))}\, dt\\
& \geq \inf\limits_\gamma \frac{1}{c}\int\limits_0^{t_\gamma} \sqrt{g_m(\kf\gamma^\prime(t),\kf\gamma^\prime(t))}\, dt
= \frac{1}{c} |\vpk(\p_\gamma)-x_\kappa|= \delta/c.
\end{align*}
Here the constant $c$ is the same as in \eqref{S2: R3}, and $\gamma:[0,1]\rightarrow \M$ runs over all piecewise smooth curves connecting $\p_\kappa$ and $\p$.
This proves the first part of inequality~\eqref{S3: equivalent of topo}.
The second part of \eqref{S3: equivalent of topo} follows in a similar manner.

Given any  point $\p\in\M$, there exists some $\kappa\in\K$ such that $\p\in \Otk$.
Set $\delta:=(1-r)/2$ and 
$\Ohk:= \psk( (r+\delta)\bar{\B}^m)$.
$ \Ohk\setminus \psk(0)$ can be identified with a collar neighborhood of $\partial\Otk$ by a  diffeomorphism 
\begin{equation*}
\Phi_\kappa: \Ohk\setminus \psk(0)\to \partial\Otk \times (0,1]: \, \p \mapsto (\psk(r\vpk(\p)/ |\vpk(\p)|) , |\vpk(\p)|/(r+\delta) ).
\end{equation*}
We can attach a neck, $\partial\Otk \times [1,2]$, to $\Ohk$ along its boundary, obtaining the new manifold 
\begin{equation*}
\Ohk\cup_{\Phi_\kappa} (\partial\Otk\ \times [1,2]),
\end{equation*}
see Figure \ref{fig}. 
Here the notation $\cup_{\Phi_\kappa}$ means that we take the union of $\Ohk$ and $\partial\Otk\ \times [1,2]$, and $\Phi_\kappa$
defines an equivalence relation between $\partial\Ohk$ and one component of the boundary of $\partial\Otk\ \times [1,2]$, i.e., $\partial\Otk\times \{1\}$.

Let $\Ohkc$ be a copy of $\Ohk$. As before, we can identify  $\Ohkc\setminus \psk(0)$ with $\partial\Otk\times [2,3)$ by a diffeomorphism $\Psi_\kappa$. 
Attach $\Ohkc$ to the other end of the neck, $\partial\Otk \times [1,2]$, to construct a new  compact manifold $\M_\kappa$ (see 
Figure \ref{fig} again), i.e.,
\begin{equation*}
\M_\kappa:=\Ohk\cup_{\Phi_\kappa} (\partial\Otk\ \times [1,2]) \cup_{\Psi_\kappa} \Ohkc.
\end{equation*}
By means of the diffeomorphism $\Psi_\kappa$, we can identify $\{\Ohk\setminus \psk(0)\}\cup \Otk \times [1,2)$ with  $\partial\Otk\times (0,2)$. 
Based on this observation, we can  introduce a differentiable structure on $\M_\kappa$. 
Such a smooth  differentiable structure on $\M_\kappa$ exists as long as there is a diffeomorphism $f: \partial\Ohk\to \partial\Otk\times\{1\}$. 
The reason for us to introduce the maps $\Phi_\kappa$ and $\Psi_\kappa$ explicitly is to show that we can construct a metric on the new manifold $\M_\kappa$ satisfying some uniform boundedness conditions on several geometric quantities of $\M_\kappa$, as shown in the condition~\eqref{S3: bd of diameter, volume, Section} below.

\begin{figure}
\centering
\includegraphics[scale=.6]{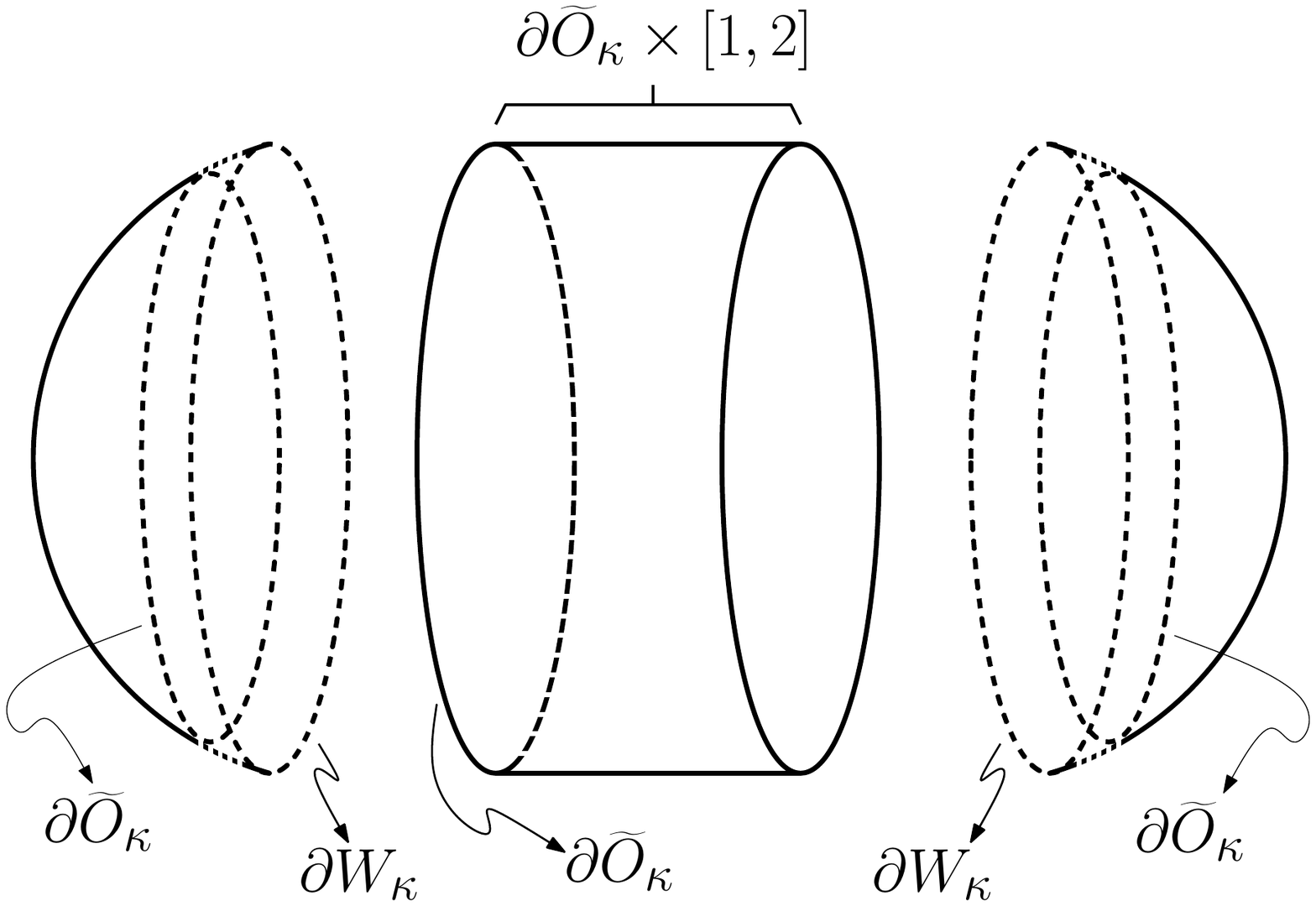}
\caption{Illustration of how $\M_\kappa$ is constructed out of gluing 
the boundary of two balls to a common cylindrical neck.}
\label{fig}
\end{figure}

Choose smooth cutoff functions 
\begin{equation*}
\chi: \mathring{W}_\kappa \rightarrow
[0,1], \quad \chi|_{\psk(\B(0,r+\delta/2))}\equiv 1, 
\end{equation*}
and 
\begin{equation*}
\zeta: \partial\Otk \times (2r/(1+r),3-2r/(1+r)) \rightarrow [0,1], \quad \zeta|_{\partial\Otk \times [(1+3r)/2(1+r),3-(1+3r)/2(1+r)]}\equiv 1.
\end{equation*}
We denote the copy of $\chi$ on $\Ohkc$ by $\chi_c$. Endow $\M_\kappa$ with the metric
\begin{equation*}
\widehat{g} = \chi g +\chi_c g + \zeta h.
\end{equation*} 
Here, $h$ is the standard product metric on  $\partial \Otk \times (0,3)$. 
 
We can choose $\kf\chi$ and $\zeta$ to be independent of $\p$ and $\kappa\in\K$. Then there exists a constant $C$ satisfying
\begin{align}
\label{S3: bd of diameter, volume, Section}
1/C\leq D(\M_\kappa) \leq C, \quad 1/C\leq  V(\M_\kappa) \leq C, \quad |K_{\widehat{g}} |_{\widehat{g}} \leq C. 
\end{align}
Here $D(\M_\kappa)$, $V(\M_\kappa)$ and $K_{\widehat{g}}$ are the diameter, volume and sectional curvature of $(\M_\kappa,\widehat{g})$, respectively. Indeed, the first two inequalities are straightforward. To verify the third inequality, we first check that, in the neighborhood $U_\kappa= \psk((r+\delta/2)\Bm\setminus (r/2)\bar{\B}^m)$ of $\partial\Otk$, the metric $g$ satisfies 
\begin{align}
\label{S3: bdd-geo of g near Otk}
\det[\widetilde{\Phi}_\kappa^* g]\geq M,\quad \|\widetilde{\Phi}_\kappa^* g|_{\Ohk}\|_{k,\infty}\leq M(k), \quad k\in \Nz, 
\end{align}
with $M$ and $M(k)$ independent of $\p$ and $\kappa\in\K$. Here 
\begin{equation*}
\widetilde{\Phi}_\kappa: \Ohk\setminus \psk(0)\to  r\mathbb{S}^{m-1}\times (0,1]: \, \p \mapsto (r\vpk(\p)/ |\vpk(\p)| , |\vpk(\p)|/(r+\delta) ).
\end{equation*} 
This assertion follows easily from (R3), (R4) and direct computation. 
Since the pull back metric of $h$ under the map 
\begin{equation*}
f: \partial\Otk\times (0,3)\to r\mathbb{S}^{m-1}\times (0,3) :\quad ({\sf q},a)\mapsto (\psk({\sf q}),a)
\end{equation*}
satisfies an analog of \eqref{S3: bdd-geo of g near Otk}, 
the boundedness of $K_{\widehat{g}}$ follows from a similar argument to the proof of part (a). 

Now it follows from \cite[Theorem~5.8]{CheegEbin08} that $\M_\kappa$ has a lower bound on its injectivity radius, i.e., $\iota(\M_\kappa)>R$ for some $R>0$. Restricting back to $\Ohk\subset\M_\kappa$, it follows from \eqref{S3: equivalent of topo} that there exists some $t^*$ independent of $\p$ such that any normal geodesic, i.e., geodesics with unit speed, starting from $\p$ cannot leave $\Ohk$ in $[0,t^*)$, and thus will not close in $\M$ within $[0,T^*)$, where $T^*= \min\{ 2 \iota(\M_\kappa), t^* \}$. This proves the injectivity radius of $\p$ is larger than or equal to $ \min\{ \iota(\M_\kappa), t^*/2 \}$.  See \cite[Section~6.6]{PPeter}.

Since the constant $C$ is independent of $\p$ and $\kappa\in\K$, so is $\iota(\M_\kappa)$. Therefore we have proved the asserted statement.
\end{proof}

\begin{remark}
As mentioned in Remark~\ref{S1: RMK}, we will give two independent proofs for Theorem~\ref{S1:Main Thm}(c).

(i) 
By the Hopf-Rinow Theorem, it suffices to show that $(\M,g)$ is complete as a metric space in the intrinsic metric induced by $g$.
As in the proof of Theorem~\ref{S1:Main Thm}(b), there exists a number $r\in (0,1)$ such that $(\Otk)_\kappa:=(\psk(r\Bm))_\kappa$ still forms an open cover of $\M$.

Assume that $(\p_l)_{l\in\N}$ is a Cauchy sequence in the intrinsic metric.
Then for any $\varepsilon\leq \delta/c$, where $\delta<1-r$ is fixed and $c$ is the constant in \eqref{S3: equivalent of topo}, there exists some $N_\varepsilon\in\N$ such that $\p_l\in\B_g(\p_n,\varepsilon)$ for all $n,l\geq N_\varepsilon$. Since $(\Otk)_\kappa$ covers $\M$, for any $n\geq N_\varepsilon$, we can find some $\kappa=\kappa(n)$ with $\p_n\in \Otk$. Then for any $l\geq N_\varepsilon$, we have
\begin{align*}
\p_l\in \B_g(\p_n,\varepsilon)\subset \B_g(\p_n,\delta/c)\subset \psk(\Bm(\vpk(\p_n),\delta))\subset\Ok.
\end{align*}
Since the maps $\vpk$ and $\psk$ are continuous between $(\Bm, g_m)$ and $(\Ok, d_g)$, it is easy to see that we actually have 
\begin{equation*}
\psk(\Bm(\vpk(\p_n),\delta))\subset\subset\Ok.
\end{equation*}
Hence there exists a convergent subsequence $\p_{l_k}\to \p $ for some $\p \in \Ok$. Since $(\p_l)_{l\in\N}$ is Cauchy, this implies that $\p_l\to \p$ in the intrinsic metric.
\medskip

(ii) Geodesic completeness of a {\URM} $(\M,g)$ can also be established directly by
considering the equations for geodesics. 
Given any $\p\in\M$ and $\Xp\in T_\p \M$,  we assume that $\p\in \Otk$. Then the equation of the geodesic  starting from $\p\in \M$ with initial velocity $\Xp\in T_\p \M$ in the local coordinate $(\Ok,\vpk)$ reads as
\begin{equation}
\label{S3: Geo eq1}
\begin{cases}
\ddot{C}^k(t)=-\Gamma^k_{ij}(C(t))\dot{C}^i(t) \dot{C}^j(t)\\
C(0)=\varphi_\kappa(\p)=:\x\\
\dot{C}(0)=d\varphi_\kappa(\p) \Xp=:\Vp,
\end{cases}
\end{equation}
where $C(t)=C^i(t) e_i\in \R^m$.
Without loss of generality, we may assume that $|\Vp|_{g_m}=1$. 

Setting $Z(t)=\dot{C}(t)$, equation \eqref{S3: Geo eq1} can be rewritten as follows.
\begin{equation}
\label{S3: Geo eq2}
\begin{cases}
\dot{C}^k(t)=Z^k(t)\\
\dot{Z}^k(t)=-\Gamma^k_{ij}(C(t))Z^i(t)Z^j(t)\\
C(0)=\x\\
Z(0)=\Vp.
\end{cases}
\end{equation}
Let $W(t):=(C(t),Z(t))$, and $F_\kappa$ be so defined that
\begin{equation*}
F_\kappa(W(t))=\left(Z^1(t),\cdots,Z^m(t);-\Gamma^1_{ij}(C(t))Z^i(t)Z^j(t),\cdots,-\Gamma^m_{ij}(C(t))Z^i(t)Z^j(t) \right).
\end{equation*}
Equation~\eqref{S3: Geo eq2} is equivalent to 
\begin{equation}
\label{S3: Geo eq3}
\begin{cases}
\dot{W}(t)=F_\kappa(W(t))\\
W(0)=(\x,\Vp).
\end{cases}
\end{equation}
By \eqref{S3: Christoffel sym}, there exists a constant $M$ uniform in all indices and $\kappa$, such that
\begin{equation*}
\|\kf\Gamma^k_{ij}\|_\infty<M . 
\end{equation*}
We fix some $\delta\in (0, 1-r)$. Then $\bar{\B}(\x,\delta) \subset \Bm.$
\begin{lem}
\label{S3: unif exist of Geo}
Given any, $(\x,\Vp)\in r\Bm \times \mathbb{S}^{m-1}$,
\eqref{S3: Geo eq3} has a unique nonextendible solution $W\in C^1(J; \Bm\times\R^m)$ on $J=J(\x,\Vp):=[0,T)$, where $T=T(\x,\Vp)\geq \tau^*:=\min\{\delta/(4\sqrt{m}), 1/(8M)\}$.
\end{lem}
\begin{proof}
By the Picard-Lindel\"of Theorem, \eqref{S3: Geo eq3} has a unique local in time solution on a maximal interval of existence $J=J(\x,\Vp):=[0,T)$ with $T=T(\x,\Vp)$. 
Suppose that for some $(\x,\Vp)$, $T(\x,\Vp)< \tau^*$. For $k=1,2,\cdots,m$, we have
\begin{align*}
Z^k(t)=Z^k(0)+\int\limits_0^t \dot{Z}^k(s)\, ds=Z^k(0) -\int\limits_0^t \Gamma^k_{ij}(C(s))Z^i(s)Z^j(s)\, ds.
\end{align*}
Let $\alpha(t):=\max\limits_{i\in\{1,\cdots,m\}}\max\limits_{s\in [0,t]} |Z^i(s)|$ for $t\in J=[0,T)$. Then we infer from the above equation that
\begin{align}
\label{S3: alpha-ineq}
\alpha(t)\leq 1+ t M \alpha^2(t), \quad t\in J.
\end{align}
Let $\alpha_{1,2}(t)$ be the solution to the equation $\alpha=1+t M \alpha^2$. It is easy to see that
\begin{equation*}
\alpha_1(t)=\frac{2}{1+\sqrt{1-4t M}}  \in [1,2), \quad \alpha_2(t)=\frac{2}{1-\sqrt{1-4t M}} \in (2,\infty).
\end{equation*}
\eqref{S3: alpha-ineq} implies that $\alpha(t)\in [0,\alpha_1(t)]\cup [\alpha_2(t),\infty)$. We will show that in fact $\alpha(t)\in [0,\alpha_1(t)]$.
Let 
\begin{equation*}
E:=\{t\in J: \alpha(t)\leq \alpha_1(t) \}.
\end{equation*}
At $t=0$, we have $\alpha(0)\leq 1=\alpha_1(0)$. Hence, $E$ is nonempty.
By the continuity of $\alpha(t)$ and $\alpha_1(t)$ in $t$, $E$ is closed in $J$. For the same reason, $J\setminus E=\{t\in J: \alpha(t)\geq \alpha_2(t)\}$ is also closed in $J$. Then $E$ is open, and thus $E=J$. Therefore
$\alpha(t)\leq \alpha_1(t)<2$ for $t\in J$. It implies that
\begin{align}
\label{S3: bd of Z}
|Z(t)|_{g_m}< 2\sqrt{m}, \quad t\in J.
\end{align} 
For any $t\in J$, this yields $|C(t)-\x|\leq 2\sqrt{m}\tau^*< \delta/2. $ We thus infer that
\begin{align}
\label{S3: bd of C}
C(t)\in \B(\x, r+\delta/2) \subset\subset \Bm, \quad t\in J.
\end{align}
\eqref{S3: bd of Z} and \eqref{S3: bd of C} contradict \cite[Theorem~7.6]{Ama90}. This proves the uniform lower bound for $T(\x,\Vp)$, i.e., $T\geq \tau^*$.
\end{proof}

Lemma~\ref{S3: unif exist of Geo} implies that, given any positive constant $C$, for any initial velocity $|\Vp|_{g_m}\leq C$, the maximal interval of existence $J=J(\x,\Vp)=[0,T(\x,\Vp))$ for the solution to \eqref{S3: Geo eq3} is uniform, i.e., $T(\x,\Vp) \geq \tau$ for some $\tau$ independent of $(\x,\Vp)$. 

Given $\p\in\M$, any geodesic $G(t)$ starting from $\p$ with initial velocity $\Xp\in T_\p \M$ fulfilling $|\Xp|_{g(p)}= 1$ satisfies equation~\eqref{S3: Geo eq1} in the local coordinate $(\Ok,\varphi_\kappa)$ with $\p\in \psi_\kappa(r\Bm)$.  
In view of (R3), $\Vp:=\kb\Xp$ fulfils $|\Vp|\leq C$ for some $C$ independent of $(\p,\Xp)$. 
Therefore, $G(t)$ exists on some $[0,T^*]$, where $T^*$ is independent of $(\p,\Xp)$. 
Since geodesics are parameterized with respect to arc length, any geodesic is infinitely extendible. This gives an alternative proof for the geodesic completeness of $(\M,g)$.
\end{remark}

\begin{remark}
The concept of {\URMs} with boundary is defined by modifying the {\em normalization} condition of the atlas $\A$ as follows. For those local patches $\Ok\cap \partial\M\neq \emptyset$, $\vpk(\Ok)=\Bm\cap \H$, where $\H$ is the closed half space $\bar{\R}^+ \times\R^{m-1}$. See \cite{Ama13, AmaAr}.
The proof of Theorem~\ref{S1:Main Thm}(a) and (c) still works for {\URMs} with non-empty boundary. 

However, the concept of positive injectivity radius for manifolds with boundary needs to be defined separately for $\p\in\mathring{\M}$ and $\p\in\partial\M$. See \cite{Per13} for precise definitions of interior and boundary injectivity radius. The idea of the proof for Theorem~\ref{S1:Main Thm}(b) can still be applied to {\URMs} in the case of non-empty boundary with necessary modifications. 
But, in this case, these manifolds are no longer geodesically complete, but only complete as a metric space. 
Therefore, roughly speaking, {\URMs} with boundary are complete as a metric space, with positive interior and boundary injectivity radius, and have bounded derivatives of the curvature tensor.
\end{remark}
\medskip


\section*{Acknowledgements}
The second and third authors would like to thank Stefan Meyer for helpful discussions.
We would like to express our appreciation to the reviewer for insightful remarks.

\end{document}